\documentclass[a4paper,12pt]{amsart}
\usepackage{amssymb}
\usepackage{bbm}
\usepackage[dvips]{graphicx}
\newtheorem{thm}{Theorem}[section]
\newtheorem*{thm*}{Theorem}
\newtheorem{lem}[thm]{Lemma}
\newtheorem*{lem*}{Lemma}
\newtheorem{cor}[thm]{Corollary}
\newtheorem*{cor*}{Corollary}
\newtheorem{prop}[thm]{Proposition}
\newtheorem*{prop*}{Proposition}
\theoremstyle{definition} 
\newtheorem{defn}[thm]{Definition}
\newtheorem*{defn*}{Definition}
\theoremstyle{remark}
\newtheorem{rem}[thm]{Remark}
\newtheorem*{rem*}{Remark}

\newtheorem*{example*}{Example}

\newtheorem*{que*}{Question}

\newcommand{\N}{\mathbb N}

\newcommand{\Z}{\mathbb Z}

\newcommand{\B}{\mathcal{B}}

\newcommand{\abs}[1]{\left\vert#1\right\vert}
\newcommand{\set}[1]{\left\{#1\right\}}
\newcommand{\eps}{\varepsilon}

\newcommand{\CB}{\mathcal{B}}

\newcommand{\ind}{\mathbbm{1}}
\renewcommand{\emptyset}{\varnothing}
\renewcommand{\tilde}{\widetilde}
\newcommand{\vect}[1]{\mathbf{#1}}

\DeclareMathOperator{\diam}{diam}

\newcommand{\PP}{\mathcal P}

\newcommand{\sd}{\bigtriangleup}

\usepackage{mathrsfs}

\renewcommand{\epsilon}{\varepsilon}

\renewcommand{\leq}{\leqslant}
\renewcommand{\geq}{\geqslant}

\title{Minimal models for actions of amenable groups}
\author{Bartosz Frej and Dawid Huczek}
\begin{document}

\begin{abstract}
We prove that on a metrizable, compact, zero-di\-men\-sio\-nal space every free action of an amenable group is measurably isomorphic to a minimal $G$-action with the same, i.e.\ affinely homeomorphic, simplex of measures.
\end{abstract}
\maketitle

\section{Introduction}
\subsection{Background}
The study has origins in the famous Jewett-Krieger theorem: to any
ergodic and invertible measure-preserving map there exists an isomorphic strictly
ergodic (i.e.~uniquely ergodic and minimal) homeomorphism. Natural further investigation concerns more precise modelling, in which given a topological dynamical system one searches for the minimal system with the same measure dynamics, i.e. having identical simplex of invariant measures with measure preserving actions being isomorphic. The adequate theorem, valid on zero-dimensional spaces even for non-invertible case, was proved by Downarowicz in \cite{D}. The next step was made by Frej and Kwa\'snicka in \cite{FK}, where the analogous result for free $\Z^d$ actions was proved. At the present paper we use the techniques developed in \cite{H} to generalize the latter to free actions of arbitrary amenable groups on zero-dimensional spaces. 

\subsection{Basic notions}
Let $G$ be a countable \emph{amenable group}, i.e. a group in which there exists a sequence of finite sets $F_n\subset G$ (called a \emph{F\o lner sequence}, or the sequence of F\o lner sets), such that for any $g\in G$ we have
\[\forall_{g\in G}\ \ \lim_{n\to\infty}\frac{\abs{gF_n\sd F_n}}{\abs{F_n}}\to 0,\]
where $gF=\set{gf:f\in F}$, $\abs{\cdot}$ denotes the cardinality of a set, and $\sd$ is the symmetric difference.

Throughout the paper we assume that $X$ is a zero-dimensional, compact metrizable space. The action of $G$ on $X$ is determined by a homomorphism from $G$ to the group all of homeomorphisms of $X$, but we will avoid introducing unnecessary notation writing $gx$ for the image of $x$ by an appropriate homeomorphism.
The action of $G$ is \emph{free} if $gx=x$ for any $g\in G$ and $x\in X$ implies that $g$ is the neutral element of $G$. The action is \emph{minimal} if for any $x\in X$ the closure of the orbit $\{gx: g\in G\}$ is equal to the whole $X$, which is equivalent to non-existence of non-trivial closed $G$-invariant subsets.

A measure $\mu$ on $X$ is $G$-invariant if $\mu(gA)=\mu(A)$ for all $g\in G$.
By $\PP_G(X)$ we denote the set  of all $G$-invariant Borel probability measures on~$X$. It is well known that in our case $\PP_G(X)$ endowed with the weak* topology is a compact, metrizable and convex subset of the space of all Borel probability measures on $X$. Every point of $\PP_G(X)$ has a unique representation as a barycenter of a certain Borel measure concentrated on the Borel set of all ergodic measures. These properties are usually abbreviated by saying that $\PP_G(X)$ is a Choquet simplex. A set $E\subset X$ is called \emph{full} if $\mu(E)=1$ for every $\mu \in \PP_G(X)$. 

\begin{defn} \label{borel_isomorphism}
We say that two dynamical systems $(X, G)$ and $(Y, G)$ with the same acting group $G$ are \emph{Borel$^*$ isomorphic} if there exists an equivariant (i.e.~commuting with the action) Borel-measurable bijection $\Phi: \tilde{X} \rightarrow \tilde{Y}$ between full invariant subsets $\tilde{X} \subset X$ and $\tilde{Y} \subset Y$, such that the conjugate map $\Phi^*: \mathcal P_G(X) \rightarrow \mathcal P_G(Y)$ given by the formula $\Phi^* (\mu) = \mu\circ\Phi^{-1}$ is a (affine) homeomorphism with respect to weak* topologies.
\end{defn}

Our main result is the following theorem.
\begin{thm} \label{main_thm}
If $X$ is a  metrizable, compact, zero-dimensional space and an amenable group $G$ acts freely on $X$ then $(X, G)$ is Borel* isomorphic to a minimal dynamical system $(Y, G)$ (with $Y$ being also metrizable, compact and zero-dimensional). 
\end{thm}

\subsection{Positive Banach density vs. syndeticity}
\begin{defn}
The set $S \subset G$ is \emph{right-syndetic} (we will briefly write \emph{syndetic}) if there exists a finite set $F\subset G$ such that $SF=G$.
\end{defn}
\begin{defn}
For $S\subset G$ and a finite $F\subset G$ denote 
\[
D_F(S)=\inf_{g\in G} \frac{|S\cap Fg|}{|F|}
\]
and
\[
D(S)=\sup\{D_F(S)\colon F\subset G, |F|<\infty\}.
\]
We call $D(S)$ the \emph{lower Banach density} of $S$.
\end{defn}
The following properties of the above notions are quite easy to prove.
\begin{prop}${}$
\begin{enumerate}
	\item If $(F_n)$ is a F\o lner sequence then $D(S)=\lim_{n\to\infty} D_{F_n}(S)$.
	\item $S$ is syndetic if and only if $D(S)>0$.
\end{enumerate}
\end{prop}

\subsection{The array representation of $(X,G)$}
Let $d_X$ denote a metric on $X$. 
Let $\Lambda=(X\cup \set{0,1,*})^\Z$, where $0$,$1$ and $*$ are additional elements which do not initially occur in $X$. We will regard elements of $\Lambda$ as bilateral sequences which will have elements of $X$ or $0$ on non-negative coordinates and elements $\set{0,1,*}$ on negative coordinates.
We define a compact metric $d$ on $X\cup \set{0,1,*}$ by
\[
d(x, y) = 
\begin{cases}
d_X(x, y) &$ for $ x,y \in X\\
\diam(X) &$ if $ x\notin X$ or $y\notin X
\end{cases} \qquad(\textrm{for } x\neq y).
\]
 We can now define a distance $d_\Lambda$ between $\vect x = (\ldots,x^{-1},x^0, x^1, \dots)$ and $\vect y = (\ldots,y^{-1},y^0, y^1, \dots)$ in $\Lambda$ by:
\[
d_{\Lambda} (\vect x, \vect y) = \sum_{i=-\infty}^{\infty} 2^{-\abs{i}} d(x^i, y^i).
\]
Note that $(\Lambda, d_{\Lambda})$ is a compact metric space. The space $\Lambda^G$ is an analogue of a multidimensional shift space. The action of $G$ on $\Lambda^G$ is defined by
\[
\left(gy\right)(h)=y(hg)\qquad\text{for every } y\in\Lambda^G.
\]

We define an array representation $\hat X$ of $X$ as a range of a map $X\ni x\mapsto \hat x\in \Lambda^G$ defined by
\[
\hat x(g)_n = \begin{cases}
gx & \text{if}\ n=0\\
0 & \text{otherwise}
\end{cases}.
\]
Levels $n\not=0$ will be used in the construction of a Borel$^*$ isomorphism aforementioned in theorem \ref{main_thm}.
It is not very hard to verify that $\hat X$ is compact and $G$-invariant and that $x\mapsto \hat x$ is in fact a topological conjugacy between $(X,G)$ and $(\hat X,G)$.

By a block in $\Lambda^G$ we mean a map $B\colon F\to\Lambda$, where $F$ is a finite subset of $G$. We define a distance between blocks $B_1$, $B_2$ on a common domain $F$ by 
\[
D(B_1,B_2) = \sup_{g\in F} d_\Lambda(B_1(g),B_2(g))
\]
and we set $D(B_1,B_2)=\diam(X)$ if the domains of $B_1$ and $B_2$ are two different subsets of $G$. We say that $B'$ is a \emph{subblock} of $B$ if the domain $F$ of $B$ contains the domain $F'$ of $B'$ and both blocks agree on $F'$. A block $B$ with the domain $F$ \emph{occurs} in $Z\subset \Lambda^G$ if there is $z\in Z$ and $g\in G$ such that $z(fg)=B(f)$ for every $f\in F$.
\begin{rem}
By Tikhonov's theorem, the set of all blocks on a fixed domain is a compact space. Moreover, compactness of $\hat X$ implies that the set of all blocks, which occur in $\hat X$ and whose domain is fixed, is compact. It follows that for every $\epsilon>0$ it contains a finite $\epsilon$-dense subset.
\end{rem}
%Similarly as in \cite{FK}, the first step of the construction of $(\tilde{X}, G)$ will be to replace $(X, G)$ by a conjugate, thus having ``the same'' simplex of measures, array system $(X^*, G)$ over the infinite alphabet $\Lambda = (X \cup \overline{\N})^{\N}$, where $\N$ denotes the set of all nonnegative integers and $\overline{\N}_0$ is the set $\N_0 \cup \{\infty\}$. Elements of $\Lambda$ and $X \cup \overline{\N}_0$ will be referred to as symbols and characters, respectively. Then, we will construct a Borel* isomorphism $\Phi$ between $(X^*, G)$ and a minimal matrix system $(\tilde{X}, G)$ over the same alphabet. The map $\Phi$ will be defined as a pointwise limit of a sequence of topological conjugacies given by block codes. 
\subsection{A few useful properties of the F\o lner sequence}
First of all note that in any amenable group there exists a F\o lner sequence with the following additional properties (see \cite{E}):
\begin{enumerate}
\item $F_n\subset F_{n+1}$ for all $n$, 
\item $e\in F_n$ for all $n$ ($e$ denotes the neutral element of $G$),
\item $\bigcup_n F_n=G$,
\item $F_n=F_n^{-1}$.
\end{enumerate}
Throughout this paper, we will assume that the F\o lner sequence which we use has those properties. 

If $F$ and $A$ are finite subsets of $G$ and $0<\delta<1$, we say that $F$ is \emph{$(A,\delta)$-invariant} if
\[\frac{\abs{F\sd AF}}{\abs{F}}< \delta,\]
where $AF=\set{af:a \in A,f\in F}$. Observe that if $A$ contains the neutral element of $G$, then $(A,\delta)$-invariance is equivalent to the simpler condition
\[\abs{AF}< (1+\delta)\abs{F}.\]

If $(F_n)$ is a F\o lner sequence, then for every finite $A \subset G$ and every $\delta>0$ there exists an $N$ such that for $n>N$ the sets $F_n$ are $(A,\delta)$-invariant. This type of invariance has the following consequence:

\begin{lem}\label{lem:invariance}
Let $F\subset G$ be a finite set. For any $\eps$ there exists a $\delta$ such that if $H\subset G$ is $(F,\delta)$-invariant then the set $H_F=\set{h\in H:Fh\subset H}$ has cardinality greater than $(1-\eps)\abs{H}$.
\end{lem}

\begin{figure}[h]
\centering
\includegraphics[width=\textwidth]{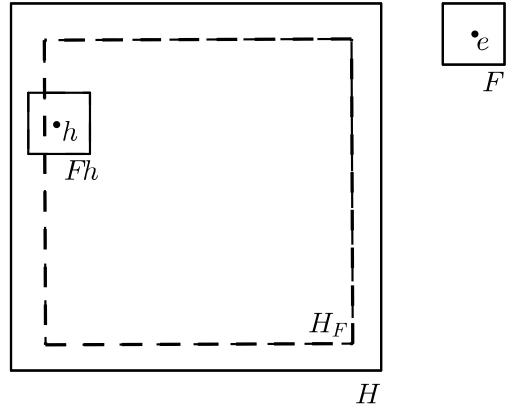}
\parbox{0.8\textwidth}{\renewcommand{\normalsize}{\footnotesize}\caption{
A picture illustrating lemma \ref{lem:invariance}. The set $H_F$ marked with a broken line has cardinality greater than $(1-\eps)\abs{H}$, where $H$ is a big square marked with a solid line.}
\label{lem17fig}
%\vspace{3mm}\hrule
}
\end{figure}

\begin{proof}
Set $\delta=\frac{\eps}{\abs{F}}$. Observe that if $h\notin H_F$ then for some $f\in F$ we have $fh \in FH\setminus H$. Obviously, for every $g\in FH\setminus H$ the number of elements $h\in H$ such that $fh=g$ for some $f\in F$ is at most $\abs{F}$. Therefore the number of elements $h\in H$ such that $fh\in FH\setminus H$ for some $f\in F$ is at most $\abs{F}\abs{FH\setminus H}\leq\abs{F}\abs{FH\sd H}<\abs{F}\delta\abs{H}=\epsilon\abs{H}$.
\end{proof}

\begin{lem}		\label{window}
Let $F\subset G$ contain the neutral element $e$. For any $l$, there exists a set $H \supset F$ such that for every $A\subset G$ and $g \in G$ if $FA\subset E \subset F^lA$ then either $Fh \subset Hg \cap E$ or $Fh \subset Hg \setminus E$ for some $h \in Hg$.
\end{lem}

\begin{figure}[h]
\centering
\includegraphics[width=\textwidth]{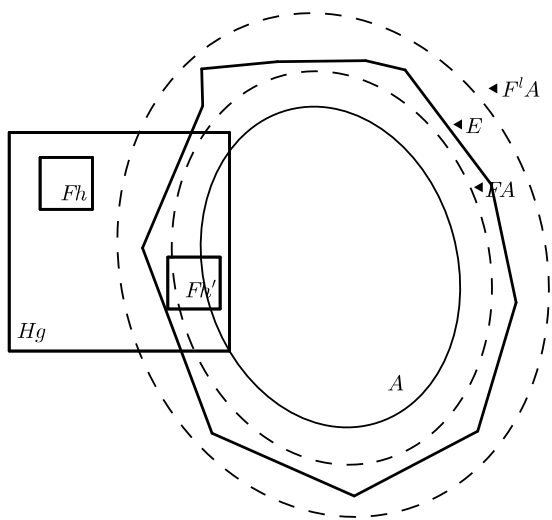}
\parbox{0.8\textwidth}{\renewcommand{\normalsize}{\footnotesize}\caption{
A picture illustrating lemma \ref{window}. The smallest elliptic set is $A$. It is surrounded by $FA$ and $F^lA$ marked by broken line, beetwen which there is a closed solid line bounding $E$. The big square is $Hg$. The two small squares show different positions of a translate of $F$: $Fh \subset Hg \setminus E$ and $Fh' \subset Hg \cap E$}
\label{lem18fig}}

%\vspace{3mm}\hrule

\end{figure}
Intuitively, this can be interpreted as follows: If we divide $H$ in two parts such that the boundary is sufficiently ``regular'' (which is expressed by the fact that the dividing set $E$ is between $FA$ and $F^lA$), then at least one of the resulting parts is regular enough to contain a translated copy of $F$.
\begin{proof}
Let $H\subset G$ be a F\o lner set such that its subset 
\[D=\set{h\in H:F^{-l}Fh\subset H}\]
has cardinality equal to at least $(1-\frac{1}{3\abs{F}^{l+1}})\abs{H}$. Suppose that the cardinality of $A\cap Hg$ is at least $\frac{1}{3\abs{F}^{l+1}}\abs{H}$. Then $A\cap Dg$ is nonempty, and for $h\in A\cap Dg$ we have $Fh\subset E$ and $Fh\subset Hg$. Otherwise, if the cardinality of $A\cap Hg$ is less than $\frac{1}{3\abs{F}^{l+1}}\abs{H}$, then we will show that $Hg\setminus F^{-1}F^lA$ (and thus also $Hg\setminus E$) has nonempty intersection with $Dg$. Indeed, observe that if $h$ is in $Hg\cap F^{-1}F^lA$, then $h\in F^{-1}F^la$ for some $a\in A$. Either $a\in Hg$ (and then $h\in F^{-1}F^l(Hg\cap A)$) or $a\notin Hg$. In the latter case $F^{-l}Fh$ is not contained in $Hg$, and therefore $h$ cannot be in $Dg$. Combining these two cases, we see that
\[Hg\cap F^{-1}F^lA\subset \left(F^{-1}F^l(Hg\cap A)\right)\cup (Hg\setminus Dg).\]
Both sets on the right have cardinality smaller than $\frac{1}{3}\abs{H}$, therefore $\abs{Hg\cap F^{-1}F^lA}<\frac{2}{3}\abs{H}$, hence $Hg\setminus F^{-1}F^lA$ is large enough to have nonempty intersection with $Dg$. If $h$ is in such an intersection, then $Fh\subset Hg$ (since $h\in Dg$) and $Fh\cap F^lA=
\emptyset$, since otherwise we would have $h\in F^{-1}{F^l}A$.

\end{proof}

\begin{rem}	\label{window_rem}
Note that the set $H$ chosen in the above lemma also satisfies the hypothesis for any subset of $F$ containing the neutral element, since the reasoning above (in particular the choice of the set $D$) can be repeated unchanged for such subsets.
\end{rem}

\subsection{Marker lemma}

\begin{lem}[Marker lemma; see \cite{H}]	\label{marker}
Let $X$ be a compact, metrisable, zero-dimensional space. Let $G$ be a group acting continuously on $X$. For every $H \subset G$ there exists a clopen set $V$ such that:
\begin{enumerate}
	\item $g(V)$ are disjoint for each $g\in H$,
	\item $\bigcup_{g\in F} g(V)=X$ for some F{\o}lner set $F$.
\end{enumerate}
\end{lem}

\begin{cor}	\label{disjoint_occur}
Let $X$ and $G$ be as above. For every $x\in X$ and every finite $T\subset G$ there is a set $C(x)\subset G$ such that 
\begin{enumerate}
	\item $Tg\cap Tg'=\emptyset$ for each pair $g,g' \in C(x)$, $g\not=g'$,
	\item lower Banach density of $C(x)$ is bounded away from zero. Even more, there is a set $F$ such that for every $x\in X$ and every $g\in G$ it holds that $C(x) \cap Fg \not=\emptyset$.
\end{enumerate}
Moreover, 
\begin{equation}	\label{equivar}
C(gx)g=C(x)
\end{equation}
for all $g\in G$ and the map $x\mapsto C(x)$ is continuous in this sense that for any F\o lner set $F$ the sets $C(x)$ and $C(x')$ agree on $F$ if $x$ and $x'$ are close enough.
\end{cor}
\begin{proof}
Let $V$ be a clopen set obtained by applying Marker lemma to the set $H=T^{-1}T$. Define $C(x)=\{g\in G\colon gx\in V\}$. Note that this immediately implies that $C(hx)h=C(x)$ for each $h\in G$. Suppose $g\in C(x)$. If $Tg \cap Tg'\not=\emptyset$, $g\not=g'$,  then $g'\in T^{-1}Tg=Hg$, so $g'g^{-1}\in H$. But $g'x\ = g'g^{-1}gx\in g'g^{-1}(V)$. Since both the neutral element $e$ and $g'g^{-1}$ are in $H$, the sets $V$ and $g'g^{-1}V$ are disjoint, so $g'x\not\in V$, meaning $g'\not\in C(x)$. This proves (1).

To prove the assertion about lower Banach density, notice that for every $g$ it holds that $\bigcup_{f\in F} gf(V)=X$ provided that $\bigcup_{f\in F} f(V)=X$. Hence, for every $x\in X$ and every $g\in G$ there is $f\in F$ such that $x\in g^{-1}f(V)$, i.e. $f^{-1}gx\in V$ or, in other words, $f^{-1}g\in C(x)$. We obtain $D_{F^{-1}}(C(x))=\inf_g \frac{|C(x)\cap F^{-1}g|}{|F^{-1}|} \geq \frac1{|F|}$. Hence $D(C(x))\geq \frac1{|F|}$.

Continuity of $x\mapsto C(x)$ stems from the fact that $V$ obtained in Marker lemma is clopen, hence for $g$ from a finite set $F$ images $gx$ simultaneously fall into $V$ or stay outside $V$ for $x$,$x'$ close enough.
\end{proof}

%======================================================================================================

\section{The model}

Let $(F_n)$ be a F\o lner sequence in $G$. 
Let $(\eps_k)$ be a summable sequence of positive numbers. For the sake of convenience, we fix a selection function $\pi$ which assigns to a pair $(D,N)$, where  $D$ is a subset of $G$ containing at least $N$ elements, a sequence $\pi(D;N)$ of $N$ \emph{different} elements of $D$. 
%We demand that this function satisfies $\pi(Dg;N)=\pi(D;N)g$ for all $g\in G$.
We will define a sequence of block codes $\Phi_k$ defined on the array representation $\hat X$ of $X$ using an inductive procedure. Certain aspects of the construction cause the first two steps to be slightly different (simpler) than subsequent ones, which is why we will begin by describing steps 1 and 2 of the construction, and then proceed to describe the procedure of step $k+1$ based on step $k$.

\smallskip

STEP 1.
Let $T_0=\set{e}$ and let $\CB_1=(B_1^1,B_2^1,\ldots,B_{N_1}^1)$ be an $\eps_1$-dense family of blocks with domain $T_0$ occurring in $\hat X$, which in this step is just an $\eps_1$-dense set of symbols occurring in $\hat X$. Apply Corollary \ref{disjoint_occur} to $T_0$, obtaining for every $x\in X$ a set $C'_1(x)$ such that the sets $T_0c$ are pairwise disjoint for $c\in C'_1(x)$ and the set $T_0C'_1(x)$ has positive lower Banach density. 
Choose $m_1$ so that for every $g\in G$, $x\in X$, the set $F_{m_1}g$ contains at least $N_1$ elements $c\in C'_1(x)$ and its cardinality satisfies  $\eps_1\abs{F_{m_1}}>N_1$. 
(To obtain $F_{m_1}$ choose $F$ whose existence is granted by (2) of cor. \ref{disjoint_occur}, find a F\o lner set which contains $N_1$ disjoint copies $Fg_1$,...,$Fg_{N_1}$ of $F$, then choose $m_1$ so that $F_{m_1}$ contains $\bigcup_{i=1}^{N_1}Fg_i$; increase $m_1$ if necessary to fulfil the size conditions). 

%For any finite subset $D$ of $F_{m_1}$ (that has cardinality at least $N_1$) there exists sequence of $N_1$ \emph{different} elements of $D$ which we will denote by $\pi_1(D)$.

Now apply Corollary \ref{disjoint_occur} to $F_{m_1}$ (playing a role of $T$), obtaining a set $C_1(x)$ for every $x\in X$. For every $x\in X$ and every $c\in C_1(x)$ the set $D=F_{m_1}\cap C'_1(x)c^{-1}$ contains an $N_1$-element sequence $\pi(D;N_1)$. Let 
\[
(d^c_1,\ldots,d^c_{N_1})=\pi(D;N_1)c\subset F_{m_1}c\cap C'_1(x).
\]
The point $\Psi_1(\hat x)$ will be created by replacing the content of $\hat x(T_0d_j^c)$ (i.e. a symbol at $\hat x(d_j^c)$) by $B_j^1$ for $j=1,\ldots,N_1$, rewriting the erased part of a trajectory of $x$ (which originally was in row $0$) to row $1$, adding the symbol $*$ in row $-1$ at coordinates $d_j^c$ and making no other changes. 
Namely, for $1\leq j\leq N_1$, let
\[D_j(x)=\set{d_j^c:c\in C_1(x)},\]
and define for $g\in D_j(x)$, where $j=1,...,N_{1}$,
\[
\Psi_1(\hat x)_{n,g}=\begin{cases} 
*&\text{if }n=-1 \\
B_j^1(0)& \text{if } n=0 \\
\hat x_{0,g}&\text{if } n=1 
\end{cases}
\]
and $\Psi_1(\hat x)_{n,g}=\hat x_{n,g}$ otherwise (by $B_j^1(0)$ we mean a symbol which lies on zero level of the block).
Continuity of the sets $C_1(x)$ and $C_1'(x)$, and their behaviour under the action of $G$, mean that the map $\Psi_1$ is continuous and commutes with the action of $G$. In fact, it may be considered a block code.
Furthermore, as $\Psi_1(\hat x)$ retains the original non-zero symbols of $\hat x$ (they were moved to row $1$ at the coordinates that were changed), it is invertible, i.e. the systems $\hat X$ and $\Psi_1(\hat X)$ are conjugate. We will now show that for any $x\in X$ every element of $\CB_1$ occurs in $\Psi_1(\hat x)$ syndetically. Indeed, the set $C_1(x)$ has positive lower Banach density (since it was obtained from the marker lemma). It follows there exists a set $E\subset G$ such that $Eg\cap C_1(x)$ is nonempty for every $g\in G$ (note that $E$ does not depend on $x$). Therefore, for every $g\in G$ the set $T_1g$, where $T_1=F_{m_1}E$, contains $F_{m_1}c$ for some $c\in C_1(x)$, which implies that every block from $\CB_1$ occurs in $\Psi_1(\hat x)$ inside $T_1g$, which is the definition of syndeticity. Moreover, we may assume that $T_1$ is symmetric, i.e. $T_1^{-1}=T_1$ (if not then replace it with $T_1^{-1}T_1$).

Finally, we choose $H_1$ using lemma \ref{window} for $F=T_1$ and we let $\Phi_1=\Psi_1$. The set $H_1$ will replace $T_1$ in the role of being a ``syndeticity constant'' for occurrence of elements of $\CB_1$ in subsequent systems we will create ($T_1$ itself would be too small). For convenience, we denote $X_1=\Phi_1(\hat X)$.

\smallskip

STEP 2.
Now let $\CB_2=(B_1^2,B_2^2,\ldots,B_{N_2}^2)$ be an $\eps_2$-dense family of blocks with domain $T_1$ occurring in $X_1$. Let $N_2$ be the cardinality of $\CB_2$.  Apply Corollary \ref{disjoint_occur} to the set $H_1^{-1} T_1$ obtaining for every $x\in X_1$ some set $C'_2(x)$. Note that for distinct $c,c'\in C'_2(x)$ each set $H_1g$ ($g\in G$) intersects at most one of $T_1c$, $T_1c'$. Indeed, if it intersected both of them then we would have $g\in H_1^{-1}T_1c \cap H_1^{-1}T_1c'$. 
%Moreover, the set $T_1C'_2(x)$ has positive lower Banach density. 
Similarly as before, choose $m_2$ so that for every $g\in G$, $x\in X_1$, the set $F_{m_2}g$ contains at least $N_2$ elements $c\in C'_2(x)$ such that $T_1c\subset F_{m_2}g$; we can also request that 
% and $N_2\abs{T_1}<\eps_2\abs{F_{m_2}}$ 
 $N_{2}\abs{T_1}^2<\eps_2\abs{F_{m_2}}$. 
(The set $F_{m_2}$ is obtained using the following procedure: choose $F$ whose existence is granted by (2) of cor. \ref{disjoint_occur}, find a F\o lner set which contains $N_2$ disjoint copies $Fg_1$,...,$Fg_{N_2}$ of $F$, then choose $m_2$ so that $F_{m_2}$ contains $\bigcup_{i=1}^{N_2} T_1Fg_i$; increase $m_2$ if necessary to fulfil the size conditions). 

As in step 1, apply Corollary \ref{disjoint_occur} to $F_{m_2}$, obtaining a set $C_2(x)$ for every $x\in X_1$. 
For every $x\in X_1$ and every $c\in C_2(x)$ the set $D=F_{m_2}\cap C'_2(x)c^{-1}$ contains a $N_2$-element sequence $\pi(D;N_2)$. Let 
\[
(d^c_1,\ldots,d^c_{N_2})=\pi(D;N_2)c\subset F_{m_2}c\cap C'_2(x).
\]
The point $\Psi_2(x)$ will be created by replacing the content of $x(T_1d_j^c)$ by $B_j^2$ for $j=1,\ldots,N_2$, rewriting the erased part of a trajectory of $x$ (which at the moment was in rows $0$ or $1$) to row $2$, adding the symbol $*$ in row $-2$ at coordinates $d_j^c$ and $1$ at other modified coordinates, and making no other changes. Namely, for $1\leq j\leq N_2$, let
\[D_j(x)=\set{d_j^c:c\in C_2(x)},\]
and define 
\[
\Psi_2(x)_{n,g}=\begin{cases}
*&\text{if }n=-2 \text{ and } g\in D_j(x)\\
1 & \text{for } n=-2, g\in T_1d \text{ for some } d\in D_j(x),\\
 &  \text{\ \ but } g\not\in D_j(x),\\
x_{N,g} &\text{for } n=2,\ g\in T_1d \text{ for some } d\in D_j(x),\\
 & \text{\ \ where } N=\max\{i\colon x_{i,g}\not=0\}\\
B_j^2(gd^{-1})(n)& \text{for } n=-1,0,1 \text{ and } g \text{ as above}
\end{cases}
\]
and $\Psi_2(x)_{n,g}=x_{n,g}$ otherwise.
Let $\Phi_2=\Psi_2\circ \Psi_1$. 
Same as before, the map $\Psi_2$ is a conjugacy between $X_1$ and its image, therefore $\Phi_2$ is a conjugacy, too. Let $X_2=\Psi_2(X_1)=\Phi_2(\hat X)$. It follows from corollary \ref{disjoint_occur}(2) that each element of $\CB_2$ occurs in $X_2$ syndetically. More precisely, there is a set $E\subset G$ such that $Eg\cap C_2(x)$ is nonempty for every $g\in G$ and therefore, putting $T_2=F_{m_2}E$, we obtain that $T_2g$ contains some $F_{m_2}c$, $c\in C_2(x)$, for every $g\in G$. We can enlarge $E$ to obtain $T_1^4\subset E$ and assume that $T_2=T_2^{-1}$, replacing $T_2$ with $T_2^{-1}T_2$ if necessary. 

Moreover, elements of $\CB_1$ also occur syndetically in $X_2$. Indeed, let $\hat x \in X_2$ and let $g\in G$. By the construction of $C'_2(x)$, the set $H_1g$ ($g\in G$) intersects at most one $T_1c$ for $c\in C'_2(x)$. By the choice of $H_1$ (see STEP 1 and Lemma \ref{window}) there exists some $h$ such that $T_1h$ is a subset of $H_1g$ that is either disjoint from all $T_1c$ or is a subset of $T_1C_2'(x)$, in which case it must be exactly one of $T_1c$, $c\in C_2'(x)$. In either case, the block $\hat x(T_1h)$ is a block occurring in $X_1$, since it was either in an area unchanged by $\Psi_2$, or was entirely replaced by content of one of the blocks in $\CB_2$ (all of which occur in $X_1$). By the construction of $\Phi_1$, this means that $\hat x(T_1h)$ contains all the blocks from $\CB_1$, which ultimately means that every block from $\CB_1$ occurs syndetically in $\hat x$.  

Finally, choose $H_2$ using lemma \ref{window} for $F=H_1^4T_2$. Note that each $H_2g$ contains some $T_2h$, thus it contains all blocks from $\CB_2$.

\smallskip

STEP $(k+1)$. Suppose we have constructed:
\begin{itemize}
\item Maps $\Phi_1,...,\Phi_k$, which are conjugacies mapping $\hat X$ to $X_k$, where $X_k=\Phi_k(\hat X) \subset \Lambda^G$,
\item $H_1,...,H_k$, $T_0,T_1,...,T_k$, which are subsets of $G$, where $T_i=T_i^{-1}$ for all $i$,
\item $F_{m_1},...,F_{m_k}$, which are selected F\o lner sets,
\item $\CB_1,...,\CB_k$, which are collections of blocks such that 
\begin{enumerate}
	\item $\CB_j$ is a collection of blocks on the domain $T_{j-1}$, $\epsilon_j$-dense in the collection of all such blocks occurring in $X_{j-1}$,
	\item $T_j \subset H_j\subset T_{j+1}\subset H_k$ and $T_1^4...T_j^4\subset T_{j+1}$ for $j=1,...,k-1$,
	\item For any $j=1,\ldots,k$, every block with domain $T_j$ in $X_j$ has every block from $\CB_j$ as a subblock,
	\item For any $j=1,\ldots,k$ and any $i\geq j$, every block with domain $H_j$ in $X_i$ has every block from $\CB_j$ as a subblock,
	\item for every $A\subset G$ and $g \in G$ if $T_{j}A\subset E \subset T_{j}^5A$ then either $H_jg\setminus E$ contains $T_{j}h$ or $H_jg\cap E$ contains $T_{j}h$ for some $h \in H_jg$, $j=1,...,k$,
	\item $|\CB_j| |T_{j-1}|^2<\eps_{j}|F_{m_{j}}|$ for $j=1,...,k$.
\end{enumerate}
\end{itemize}
We will make use of the following:
\begin{quote}
\emph{Enlarging algorithm}\\
Let $T$ be a set in $G$. Fix $x\in X$ and $k \in \N$. Let $T'$ be the union of $T$ and all sets of the form $T_kg$ (where $g$ is any point such that $x(-k,g)=*$) that intersect $T_kT$. We put 
\[
E^1(T,x,k)=T_kT'.
\]
Then inductively we define 
\[
E^{j+1}(T,x,k)=E^1(E^{j}(T,x,k),x,k-j)
\]
for $1\leq j < k$ and
\[
E(T,x,k)=E^k(T,x,k).
\]
Note that $T_kT\subset E(T,x,k)\subset T_1^4T_2^4...T_k^4T$. More precisely, 
\[
T_jE^{k-j}(T,x,k) \subset E(T,x,k)  \subset T_1^4...T_j^4E^{k-j}(T,x,k) %\subset H_j^{4j}E^j(T,x,k)
\]
for $1\leq j<k$. Indeed, by the inductive assumption (2), we can even say that 
\[T_jE^{k-j}(T,x,k) \subset E(T,x,k)  \subset T_j^5E^{k-j}(T,x,k)\]
\begin{figure}[h]
\centering
\includegraphics[width=\textwidth]{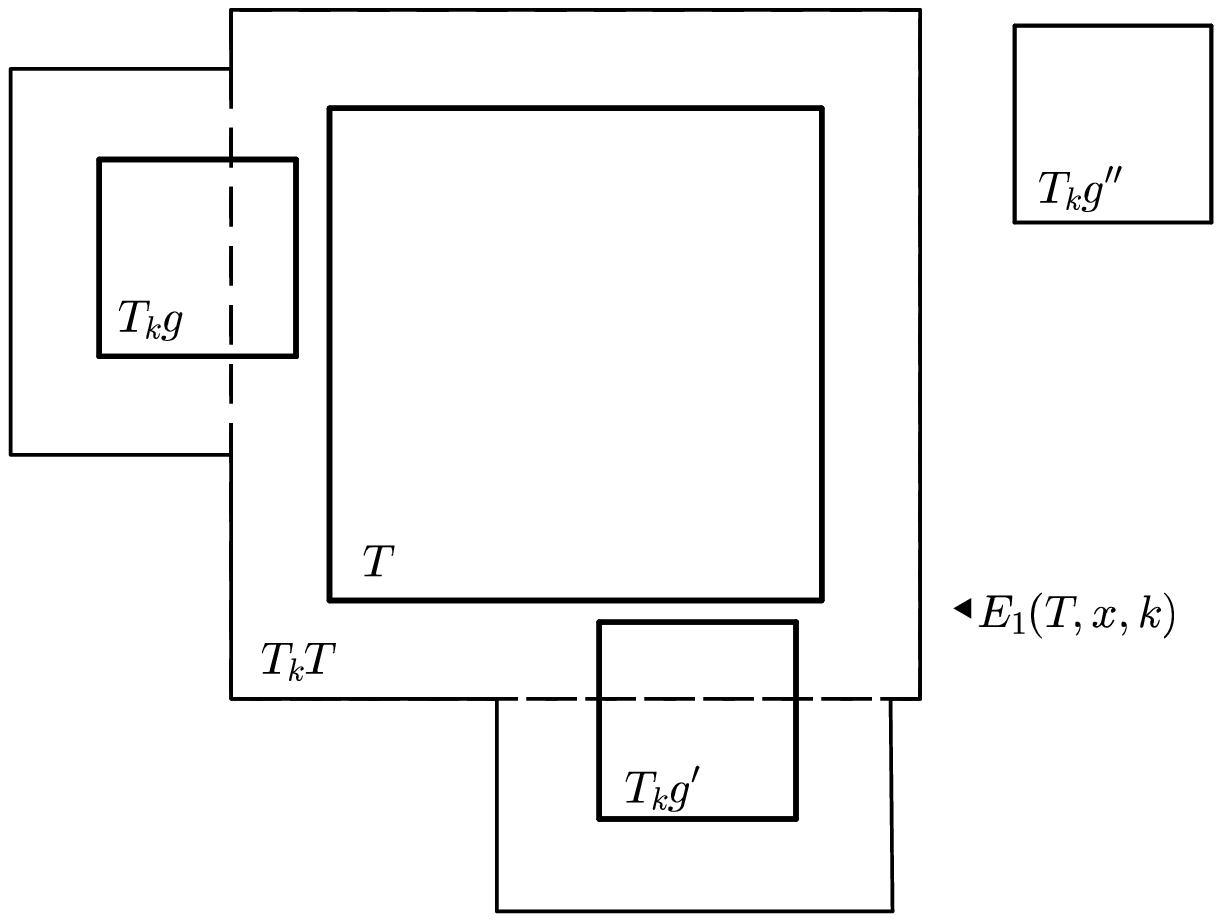}
\parbox{0.8\textwidth}
{\renewcommand{\normalsize}{\footnotesize}\caption{}
\label{enlarge}
%\vspace{3mm}\hrule
}
\end{figure}
\end{quote}
We proceed with the construction.
 Let $\CB_{k+1}=(B_1^{k+1},B_2^{k+1},\ldots,B_{N_{k+1}}^{k+1})$ be an $\eps_{k+1}$-dense family of blocks with domain $T_k$, occurring in $X_k$. Let $N_{k+1}$ be the cardinality of $\CB_{k+1}$. 
 %Choose $H_k$ using lemma \ref{window} for $F=(H'_{k-1})^3H'_k$. 
 Let $\overline{T}=T_1^4T_2^4...T_{k-1}^4T_k$ (note that by our inductive assumptions, $T_k\subset \overline T \subset T_k^2$).
Apply Corollary \ref{disjoint_occur} to the set $H_k^{-1} \overline{T}$ obtaining for every $x\in X_k$ some set $C'_{k+1}(x)$. Note that for distinct $c,c'\in C'_{k+1}(x)$ each set $H_kg$ ($g\in G$) intersects at most one of $\overline{T}c$, $\overline{T}c'$. Moreover, $C'_{k+1}(x)$ has positive Banach density, therefore so do $\overline{T}C'_{k+1}(x)$ and $T_kC'_{k+1}(x)$. Similarly as before, choose $m_{k+1}$ so big that for every $g\in G$, $x\in X$, the set $F_{m_{k+1}}g$ contains at least $N_{k+1}$ elements $c\in C'_{k+1}(x)$ such that $T_{k}c\subset F_{m_{k+1}}g$; we can assume that 
%$N_{k+1}\abs{T_{k}}<\eps_{k+1}\abs{F_{m_{k+1}}}$.
 $N_{k+1}\abs{T_{k}}^2<\eps_{k+1}\abs{F_{m_{k+1}}}$. 

Apply again Corollary \ref{disjoint_occur} to $F_{m_{k+1}}$, obtaining for every $x\in X_k$ a set $C_{k+1}(x)$. 

Smilarly as in previous steps, for every $x\in X_k$ and every $c\in C_{k+1}(x)$ the set $D=F_{m_{k+1}}\cap C'_{k+1}(x)c^{-1}$ contains a $N_{k+1}$-element sequence $\pi(D;N_{k+1})$. Let 
\[
(d^c_1,\ldots,d^c_{N_{k+1}})=\pi(D;N_{k+1})c\subset F_{m_{k+1}}c\cap C'_{k+1}(x).
\]
For rows from $-k$ to $k$ we replace the current content of $x$ within the domain $E:=E(T_{k}d_j^c,x,k-1)$ by the content of 
%$\overline{B_j^{k+1}}(d_j)$, where $\overline{B_j^{k+1}}(d_j)$ is 
any block with domain $E$ occurring in $X_k$ and containing $B_j^{k+1}$ as a subblock on coordinates from $T_{k}d_j^c$.

For $1\leq j\leq N_{k+1}$, let 
\[
D_j(x)=\set{d_j^c:c\in C_{k+1}(x)}.
\]
Denote $N(x,g)=\max\{j\colon x_{j,g}\not=0\}$ and define:
\[\Psi_{k+1}(\hat x)_{n,g}=\begin{cases}
*&\text{for }n=-(k+1)\textrm{ and } 
g \in D_j(x),\\
1 & \text{for } n=-(k+1), g\in T_kd \\
 &  \text{\ \ for some } d\in D_j(x),\text{ but } g\not\in D_j(x),\\
x_{N(x,g),g} &\text{for } n={k+1} \\
 & \text{\ \ and } g\in T_kd \text{ for some } d\in D_j(x),\\
B_j^{k+1}(gd^{-1})(n)& \text{for } n=-k,...,k \text{ and } g\text{ as above,}  \\
 %, \text{ where } \\
 %& d_j\in \pi(F_{m_{k+1}}\cap C'_{k+1}(x)c^{-1};N_{k+1})c, c\in C_{k+1}(x),\\
 & \text{\ \ with the rest of } E(T_{k}d_j,x,k-1)\\
 & \text{\ \ completed to a block}\\
 & \text{\ \ occuring in }\Psi_k(\hat X)
\end{cases}\]
and $\Psi_{k+1}(\hat x)_{n,g}=x_{n,g}$ otherwise.

Let $\Phi_{k+1}=\Psi_{k+1}\circ \Phi_{k}$, $X_{k+1}=\Phi_{k+1}(\hat X)$. Again it follows from the construction that $\Psi_{k+1}$ and hence also $\Phi_{k+1}$ are conjugacies. 

Since $C_{k+1}(x)$ was chosen with use of corollary \ref{disjoint_occur}, there is a set $E$ such that for every $x\in X_k$ and every $g\in G$ it holds that $C_{k+1}(x) \cap Eg \not=\emptyset$.
Enlarging $E$ we may assume that $T_1^4...T_k^4\subset E$.
Let $T_{k+1}=F_{m_{k+1}}E$. Replace $T_{k+1}$ with $T^{-1}_{k+1} T_{k+1}$ if it was not symmetric. Then $T_{k+1}g$ contains some $F_{m_{k+1}}c$, where $c\in C_{k+1}(x)$, for every $g\in G$. Consequently, any block with domain $T_{k+1}$ in $X_{k+1}=\Psi_{k+1}(X_k)$ has every block from $\CB_{k+1}$ as a subblock.

Choose $H_{k+1}$ using lemma \ref{window} for $F=T_{k+1}$ and $l=5$. Note that each $H_{k+1}g$ contains $T_{k+1}g$, so all blocks from $\CB_{k+1}$ occur in each $H_{k+1}g$. Also, all changes in step $k+1$ are on coordinates covered by sets of the form $T^2_{k}g$, where $g$ are such that $x(-(k+1),g)=*$.

%OLD VERSION
%Now we will show that also for each $j=1,...,k$ all blocks from $\CB_j$ occur in any block with domain $H_jg$ ($g\in G$) in $X_{k+1}$. Fix $g\in G$, $x\in X_k$ and $c\in C_{k+1}(x)$. Let $E=E(T_{k}d_j,x,k-1)$ for some $d_j\in \pi(F_{m_{k+1}}\cap C'_{k+1}(x)c^{-1};N_{k+1})c$. 

%Either $H_jg$ includes some ${T}_jh$ contained in $E$ or is disjoint from it. 
%END OLD VERSION

Now we will show that also for each $j=1,...,k$ all blocks from $\CB_j$ occur in any block with domain $H_jg$ ($g\in G$) in $X_{k+1}$. Fix $g\in G$, $x\in X_k$ and $c\in C_{k+1}(x)$. For any set $E$ of the form $E(T_{k}d,x,k-1)$ %(where $d\in \pi(F_{m_{k+1}}\cap C'_{k+1}(x)c^{-1};N_{k+1})c$)
(where $d\in D_j(x)$ for some $j$), there are three possibilities (indeed there is at most one such set for which one of the latter two is true): The set $H_jg\cap E$ can be empty, can be the whole of $H_jg$ or it can be something else. In the first two cases, the construction of $\Psi_{k+1}$ implies that the block with domain $H_jg$ is a block occurring in $X_k$, and thus contains all blocks from $\CB_j$ by the inductive assumption (4). Thus we need only concern ourselves with the case when $H_jg\cap E$ is not empty, but $H_jg$ is not a subset of $E$ (let us reiterate that this can happen for at most one $E$ among the sets we consider). The construction of $E$ and our inductive assumptions imply that there is some $h$ for which either $T_jh\subset H_jg\setminus E$ or $T_jh\subset H_jg\cap E$. Observe that $T_jh$ intersects no coordinates that were modified after step $j$: if it did, then it would intersect a set of the form $T_{k'}g'$ (for $j<k'\leq k$ and $g'$ such that $x(-k',g')=*$), but the enlarging algorithm would then cause $H_jg$ to be a subset of $E$. It follows that the block with domain $T_jh$ is a block occurring in $X_j$, thus it (as well as the larger block with domain $H_jg$) contains all blocks from $\CB_j$.
%
%If $H_jg$ is disjoint from all such $E$ then it was not changed at this level, so it has correct structure of $\CB_j$ blocks. 
%If not then it intersects exactly one such $E$. By induction assumption, either $H_jg\setminus E$ contains some $T_{j}h$ or $H_jg\cap E$ contains some $T_{j}h$. Anyway, such $T_{j}h$ is a block occurring in $X_j$ so it contains correct structure of $\CB_j$ blocks.

%=================================================================================================================

\section{The isomorphism}

Let $\tilde{X}^e_k$ be a subset of all $x\in X$ such that $\Phi_{k+1}(x)_e\not=\Phi_k(x)_e$. It is exactly the set of points whose $\Phi_k$-image   was modified on coordinate $e$ by $\Psi_{k+1}$.
By the ergodic theorem (see e.g. \cite{L}), for each $G$-invariant ergodic measure $\mu$ on $X$ it holds that 
\[
\mu(\tilde{X}^e_k) = \lim_{n\to\infty} \frac1{|F_{n}|}\sum_{g\in F_{n}} \ind_{\tilde{X}^e_k}(gx) \qquad \textrm{in } L^1(\mu).
\]
Let $F_{m_k}$ denote the F\o lner sets selected during the construction.
\newcommand{\Fk}{F_{m_{k+1}}}
The code $\Psi_{k+1}$ introduces changes only on at most $|\CB_{k+1}||T_k|^2 < \epsilon_{k+1}|F_{m_{k+1}}|$ coordinates of each block $x(F_{m_{k+1}}c)$, where $c\in C_{k+1}(x)$. We will estimate the number of changes in a block $x(F_n)$ for sufficiently large $n$. We choose $n$ so big that $F_n$ is $(F_{m_{k+1}},\eps_{k+1})$-invariant. For each $c\in C_{k+1}(x) \cap F_n$ the set $\Fk c$ is a subset of $\Fk F_n$, the latter having less than $(1+\eps_{k+1})F_n$ elements. Since $F_{m_k}c \cap F_{m_k}c' = \emptyset$ for $c\not=c'$, $c,c'\in C_{k+1}(x)$, the set $F_n$ may contain at most $\frac{(1+\eps_{k+1})|F_n|}{|\Fk|}$ elements of $C_{k+1}(x)$. 
Moreover, $F_n$ may intersect $\Fk c$ for some $c\in C(x)\setminus F_n$, but then $c\in \Fk^{-1}F_n=\Fk F_n$.
Total number of changes made by $\Psi_{k+1}$ is thus less than
\[
\left(\frac{(1+\eps_{k+1})|F_n|}{|\Fk|}  + \eps_{k+1}|F_n|\right)\cdot |\CB_{k+1}||T_k|^2< \eps_{k+1}(1+2\eps_{k+1})|F_n|.
\] 
Therefore, $\mu(\tilde{X}^e_k)\leq\eps_{k+1}(1+2\eps_{k+1}) < 2\epsilon_{k+1}$.

Let $\tilde{X}\subset X$ be the set of such $x\in X$ that for each $g\in G$ the sequence $\Phi_k(x)_g$ is eventually constant.
\[
\tilde{X}=\bigcap_{g\in G} \bigcup_{k=1}^\infty  \bigcap_{j=k}^\infty X\setminus g(\tilde{X}^e_j) = X \setminus \bigcup_{g\in G} \bigcap_{k=1}^\infty  \bigcup_{j=k}^\infty g(\tilde{X}^e_j)
\]
Since $\epsilon_k$ is summable, $\mu(\bigcup_{j=k}^\infty g(\tilde{X}^e_j))$ converges to zero when $k$ goes to infinity and therefore $\tilde{X}$ has full measure. Consequently, 
\[
\Phi(x)=\lim_{k\to\infty} \Phi_k(x)
\]
is defined on a full subset of $X$ (almost everywhere for each $G$-invariant measure). Let $Y$ be the closure of $\overline{\Phi(\tilde{X})}$ in $\Lambda^G$. Clearly, it is $G$-invariant. By essentially the same argument as in \cite{D} and \cite{FK} one can prove the following proposition.
\begin{prop}
\begin{enumerate}
	\item $\Phi(\tilde{X})$ is a full subset of $Y$.
	\item $\Phi$ is an equivariant Borel-measurable bijection onto a full set.
	\item $\Phi^*$ is an affine homeomorphism between simplices of invariant measures on $X$ and $Y$.
\end{enumerate}
\end{prop}
\begin{proof}
%Consider the set
%\[
%\tilde{Y}=\{y\in Y: \forall g\in G\ \max\{n:y_{n,g}\not=0\}<\infty\}\}.
%\]
%The set has full measure. Indeed, for any $y\in Y$ consider $C_k^*(y)=\{g\in G: y_{-k,g}=*\}$. Clearly, $\tilde{Y}=\{y:\bigcap_{k=1}^\infty C_k^*(y)=\emptyset\}$.
%For each $x$ the set $C_k^*(\Phi(x))$ is exactly the set of such $d\in C'_k(x)$ that  $\Phi_k(x)_g=\Phi_{k-1}(x)_g$ for $g\not\in \bigcup_d E(T_{k-1}d,x,k-2)$.

For $y\in Y$ take $(y_n)\subset \Phi(\tilde X)$, $y_n=\Phi(x_n)$, converging to $y$. Let $Y_k^e$ be the set of all this $y\in Y$ for which $y_{-k,e}\not=0$. For sufficiently large $n$ on the set $F_m$ the positions of symbols in $y$ from $-k$ up to level $k$ coincide with those in $y_n$. Non-zero symbols appear on the $-k$th level of $y_n$ only at coordinates on which $x_n$ was modified by $\Psi_{k}$. Hence, for any $G$-invariant measure $\nu$ on $Y$ the number $\nu(Y_k^e)$ is estimated similarly to the case of $\tilde{X}^e_k$: the frequency of non-zero symbols on $-k$th level of $x(F_{m_k})$ is bounded by $\eps_{k}$, and again by the ergodic theorem
\[
\nu(Y^e_k) = \lim_{m\to\infty} \frac1{|F_{m}|}\sum_{g\in F_{m}} \ind_{Y^e_k}(gx)<\eps_k.
\]
Hence the set
\[
\tilde{Y}=\bigcap_{g\in G} \bigcup_{k=1}^\infty  \bigcap_{j=k}^\infty Y\setminus g(Y^e_k) = Y \setminus \bigcup_{g\in G} \bigcap_{k=1}^\infty  \bigcup_{j=k}^\infty g(Y^e_k),
\]
is a full set. Notice that for each $y\in\tilde{Y}$ and every $g\in G$ the elements $y_{n,g}$ are eventually zero (both when $n\to\infty$ and $n\to-\infty$).

We need to show that $\Phi$ maps $\tilde{X}$ onto $\tilde{Y}$. Choose $y\in\tilde{Y}$. Let $x\in X$ be the element lying on the last non-zero level of $y_{\cdot,e}$. Let $y=\lim_{n\to\infty}y_n$, $y_n\in \Phi(\tilde X)$, $\Phi(x_n)=y_n$. Consider some F\o lner set $F_k$. On $F_k$ the point $y$ has zeros on levels higher than some $N$. For $n$ sufficiently large, all $y_n$ have zeros on levels higher than $N$. Hence all coordinates of $x_n$ belonging to $F_k$ were changed at most $N$ times. It means that $\Phi(x_n)_{l,g}=\Phi_N(x_n)_{l,g}=\Phi_{N+i}(x_n)_{l,g}$ for $g\in F_k$ and $i\geq 1$. Since $\Phi_n$ is a block code, it is continuous, thus $\Phi_N(x)=\lim_{n\to\infty}\Phi_{N}(x_n)=\lim_{n\to\infty}\Phi_{N+i}(x_n)=\Phi_{N+i}(x)$. Consequently, in $x$ coordinates $g\in F_k$ are modified at most $N$ times by codes $\Phi_N$. Since $F_k$ is an arbitrarily big F\o lner set, $x\in\tilde X$. At the same time, the content of $x$ at these coordinates coincides with the content of $y=\lim_n y_n$, so $y=\Phi(x)$.

Injectivity follows easily from the fact that the last non-zero levels of $\Phi(x)$ contain the orbit of $x$ under the action of $G$.

The map $\Phi$ is measurable, because it is a limit of (continuous) block codes. Both spaces are compact, hence standard Borel, so the inverse is automatically measurable. This ends the proof of (1) and (2).

The map $\Phi^*$ is obviously affine, its bijectivity follows from the fact that $\Phi$ is a bijection between full sets, so it suffices to show that it is continuous (note that sets of invariant measures are compact). We skip this argument, because it is almost identical to the one for the $\Z^d$ case presented in \cite{FK}. 
\end{proof}

%=================================================================================================================

\section{Minimality}
%Let $b$ be a block defined on $F$.
%\begin{defn}
%A block $b$ \emph{occurs syndetically} in $x\in\Lambda^G$ if
%there is a finite set $S\subset G$ such that for every $g\in G$ there is $s\in S$ such that $x(Fsg)=b(F)$. The set $S$ will be called \emph{the set of syndeticity}.
%\end{defn}
%\begin{rem}
%If $LBD\set{c\in g\colon x(Fc)=b(F)}>0$ then $b$ occurs syndetically in $x$.
%\end{rem}
Minimality of the system $(Y,G)$ follows from the following lemma.
\begin{lem}
Let $(Y,G)$ be an array system, $Y\subset\Lambda^G$. Let $\B_Y$ be the collection of all blocks occuring in $Y$ and let $\B_Y'\subset \B_Y$ be a countable collection of blocks with the following property:\\
for every $\epsilon>0$ and every $B\in \B_Y$ there exists $B'\in\B_Y'$ such that $D(B,B'')<\epsilon$ for some subblock $B''$ of $B'$.

If there exist a dense subset $Y'$ of $Y$ consisting of elements $y$ in which every $B \in \mathcal B'_{Y}$ occurs syndetically then the symbolic system $(Y, \sigma)$ is minimal.
\end{lem}

\begin{proof}
Let us metrize the topology in $\Lambda^G$, $G=\{g_1,g_2,...\}$, by 
\[
\rho(x,y) = \max_{n\in\N} \frac{d_\Lambda(x(g_n),y(g_n))}n.
\]
We fix $x\in Y$ and aim to prove that the orbit $\{gx\colon g\in G\}$ is dense in $Y$. Let $\epsilon>0$ and let $y\in Y$. We will show that $\rho(gx,y)<\epsilon$ for some $g\in G$. Let $y'\in Y'$ satisfy $\rho(y,y')<\frac\epsilon3$. Choose $N\in\N$ such that $\frac1N\diam\Lambda<\frac\epsilon3$ and let $F$ be a F{\o}lner set containing $\{g_1,...,g_N\}$. 
Denote by $y(F')$ a block being the restriction of $y'$ to the domain $F$. It follows from the assumption that we can find a block $B'\in B'_Y$ such that $D(y'(F),B'')<\frac\epsilon3$ for some subblock $B''$ of $B'$. Note that $B''$ occurs in each element $z$ of $Y'$ syndetically, i.e.~there is a finite set $H\subset G$ such that $B''$ (possibly translated) is a subblock of $z(Hg)$ for each $g\in G$. In particular, it is so for $g$ being the neutral element, i.e. there is $h$ such that $Fh\subset H$ and $B''=z(Fh)=hz(F)$. It follows that $h\in H$, because F\o lner sets contain the neutral element.

Since $H$ is finite, we can approximate a point $x$ (fixed at the beginning) by $x'\in Y'$ so that $\rho(x,x')<\delta$, where delta is small enough to ensure that 
$\rho(hx,hx')<\frac\epsilon3$ for all $h\in H$. For some $h_0\in H$ we have $h_0x'(F)=B''$, hence for each $f\in F$ it holds that 
\[
d_\Lambda((h_0x')(f),y'(f)) \leq D(h_0x'(F),y'(F))=D(B'',y'(F))<\frac\epsilon3.
\] 
Since $F\supset\{g_1,...,g_N\}$, for $f\not\in F$ we have $\frac1N d_\Lambda((h_0x')(f),y'(f)) <\frac\epsilon3$, so $\rho(h_0x',y') \leq \frac\epsilon3$
and
\[
\rho(h_0x,y) \leq \rho(h_0x,h_0x') + \rho(h_0x',y') + \rho(y',y) < \epsilon.
\]
\end{proof}
Note that in our case the union $\bigcup_{k=1}^\infty \CB(X_k)$ is a collection of blocks with the desired property. Indeed, each block occurring in $\Phi(\tilde X)$ occurs already in one of $X_k$, because each coordinate of $\tilde{X}$ was modified only finitely many times. Note that the distance between blocks majorizes the distance between their subblocks sharing the same domain. Since $\Phi(\tilde X)$ is dense in $Y$, the main theorem \ref{main_thm} is proved.

%======================================================================================================

\end{document}